\documentclass[12pt]{amsart}

\author{Simeon Ball and Michel Lavrauw} \thanks{
The first author acknowledges the support of the project MTM2014-54745-P of the Spanish {\em Ministerio de Econom\'ia y Competitividad.}}

\title{Planar arcs}

 \usepackage{amsmath}
\usepackage{amssymb}
\usepackage{amsthm}
\usepackage{graphicx}
\usepackage{amscd}

\newtheorem{theorem}{Theorem}
\newtheorem{lemma}[theorem]{Lemma}

 \newtheorem{corollary}[theorem]{Corollary}
 
 \newtheorem{conjecture}[theorem]{Conjecture}
\newtheorem{example}{Example}

\begin{document}

\baselineskip=17pt

\date{\today}

\maketitle

\begin{abstract}
Let $p$ denote the characteristic of ${\mathbb F}_q$, the finite field with $q$ elements. 
We prove that if $q$ is odd then an arc of size $q+2-t$ in the projective plane over ${\mathbb F}_q$, which is not contained in a conic, is contained in the intersection of two curves, which do not share a common component, and have degree at most $t+p^{\lfloor \log_p t \rfloor}$, provided a certain technical condition on $t$ is satisfied. 

This implies that if $q$ is odd then an arc of size at least $q-\sqrt{q}+\sqrt{q}/p+3$ is contained in a conic if $q$ is square and an arc of size at least $q-\sqrt{q}+\frac{7}{2}$ is contained in a conic if $q$ is prime. This is of particular interest in the case that $q$ is an odd square, since then there are examples of arcs, not contained in a conic, of size $q-\sqrt{q}+1$, and it has long been conjectured that if $q \neq 9$ is an odd square then any larger arc is contained in a conic. 

These bounds improve on previously known bounds when $q$ is an odd square and for primes less than $1783$. The previously known bounds, obtained by Segre \cite{Segre1967}, Hirschfeld and Korchm\'aros \cite{HK1996} \cite{HK1998}, and Voloch \cite{Voloch1990b} \cite{Voloch1991}, rely on results on the number of points on algebraic curves over finite fields, in particular the Hasse-Weil theorem and the St\"ohr-Voloch theorem, and are based on Segre's idea to associate an algebraic curve in the dual plane containing the tangents to an arc. In this paper we do not rely on such theorems, but use a new approach starting from a scaled coordinate-free version of Segre's lemma of tangents.

Arcs in the projective plane over ${\mathbb F}_q$ of size $q$ and $q+1$, $q$ odd, were classified by Segre \cite{Segre1955b} in 1955. In this article, we complete the classification of arcs of size $q-1$ and $q-2$. 
\end{abstract}

\section{Introduction.}


Let ${\mathbb F}_q$ denote the finite field with $q$ elements and let $p$ denote the characteristic of ${\mathbb F}_q$. 

Let $\mathrm{PG}(k-1,q)$ denote the $(k-1)$-dimensional projective space over ${\mathbb F}_q$.

An {\em arc} of $\mathrm{PG}(k-1,q)$ is a set of points, any $k$ of which span the whole space. An arc is {\em complete} if it cannot be extended to a larger arc. In this article we will be interested in arcs in $\mathrm{PG}(2,q)$, which we call {\em planar arcs}. A planar arc is defined equivalently as a set of points, no three of which are collinear. One can project any higher dimensional arc $A$ to a planar arc of size $|A|-k+3$, by choosing any $k-3$ points of the arc and quotienting by the subspace that they span. Therefore, the results contained in this article have implications for all low-dimensional arcs.

Arcs not only play an important role in finite geometry but also forge links to other branches of mathematics. The matrix whose columns are vector representatives of the points of an arc, generates a linear maximum distance separable code, see \cite[Chapter 11]{MS1977} for more details on this. Other areas in which planar arcs play a role include the representation of matroids, see \cite{Oxley1992}, Del Pezzo surfaces over finite fields, see \cite{BFL2016}, bent functions, see \cite[Chapter 7]{Mesnager2016} and pro-solvable groups, see \cite{CL2012}.

In 1955, Beniamino Segre published the article \cite{Segre1955a}, which contains his now celebrated theorem that if $q$ is odd then a planar arc of size $q+1$ is a conic. He went further in his 1967 paper \cite{Segre1967} and considered planar arcs of size $q+2-t$ and proved that the set of tangents, when viewed as a set of points in the dual plane, is contained in a curve of small degree $d$. Specifically, if $q$ is even then $d=t$ and if $q$ is odd then $d=2t$.

His starting point, which will also be our starting point, was his lemma of tangents. Lemma~\ref{segre} is a simplification of the coordinate-free version of Segre's original lemma of tangents which first appeared in \cite{Ball2012}.

Here, we do not apply Segre's lemma of tangents in the dual setting, nor combine it with interpolation, as was the approach in \cite{Ball2012}. Our aim here is to prove Theorem~\ref{main}, which maintains that there is a polynomial $F(X,Y)$, where $X=(X_1,X_2,X_3)$ and $Y=(Y_1,Y_2,Y_3)$, homogeneous of degree $t$ in both $X$ and $Y$, which upon evaluation in one of the variables at a point $a$ of the arc or at least a large subset of the arc, factorises into $t$ linear forms whose kernels are precisely the tangents to the arc at $a$.

This allows us to prove the following theorem.

\begin{theorem} \label{twocurves}
Let $A$ be a planar arc of size $q+2-t$, $q$ odd, not contained in a conic. 

If $A$ is not contained in a curve of degree $t$ then it is contained in the intersection of two curves of degree at most $t+p^{\lfloor \log_p t \rfloor}$ which do not share a common component. 

If $A$ is contained in a curve $\phi$ of degree $t$ and
\begin{equation} \label{cond}
p^{\lfloor \log_p t \rfloor}(t+\tfrac{1}{2}p^{\lfloor \log_p t \rfloor}+\tfrac{3}{2}) \leqslant \tfrac{1}{2}(t+2)(t+1)
\end{equation}
then there is another curve of degree at most $t+p^{\lfloor \log_p t \rfloor}$ which contains $A$ and shares no common component with $\phi$.
\end{theorem}

We do not believe that the condition (\ref{cond}) in Theorem~\ref{twocurves} is necessary. For the values of $t$ where (\ref{cond}) is not satisfied one can remove a few points of the arc so that it is satisfied and derive the same conclusion for a large subset of the arc. 


It is worth remarking that Theorem~\ref{twocurves} may hold for $q$ even as well, if we replace the condition ``not contained in a conic" with ``not contained in an arc of maximum size". Arcs of maximum size for $q$ even are of size $q+2$ and are known as hyperovals. In general, these are not contained in low degree curves and large subsets of these hyperovals are also not contained in low degree curves. Indeed, almost the opposite occurs. In \cite{CS2015}, Caullery and Schmidt prove that if $S=\{(f(x),x,1) \ | \ x \in {\mathbb F}_q \} \cup \{(1,0,0),(0,1,0) \}$ is an arc of size $q+2$ and $f$ has degree less than $\frac{1}{2}q^{\frac{1}{4}}$, then $f(x)=x^6$ or $f(x)=x^{2^k}$ for some positive integer $k$. There are many other infinite families of hyperovals known, see \cite{Cherowitzo1996} for a survey.

In the following sections we will prove various results, which are corollaries to Theorem~{\ref{twocurves}. We will then go on to prove the aforementioned Lemma~\ref{segre}, Theorem~\ref{main} and finally Theorem~\ref{twocurves}.

\section{The second largest complete arc and the $\sqrt{q}$ conjecture.} \label{section:rootq}

In 1947, Bose \cite{Bose1947} was the first to observe that the largest planar arc has size $q+1$ if $q$ is odd and $q+2$ if $q$ is even. The size of the second largest complete arc remains an open question in general.

There are examples of complete arcs of size $q-\sqrt{q}+1$ in $\mathrm{PG}(2,q)$ when $q$ is square, first discovered by Kestenband, see \cite{Kestenband1981}. These arcs are the intersection of two Hermitian curves of degree $\sqrt{q}+1$,
$$
x_1^{\sqrt{q}+1}+x_2^{\sqrt{q}+1}+x_3^{\sqrt{q}+1}=0 \ \ \ \mathrm{and} \ \ \ \sum_{i,j=1}^3 x_i^{\sqrt{q}}x_j h_{ij}=0,
$$
where $h_{ij}=h_{ji}^{\sqrt{q}}$ and the  characteristic polynomial of the matrix $H=(h_{ij})$ is irreducible over ${\mathbb F}_q$. It was this construction that motivated us to try and prove that arcs, not contained in a conic, are contained in the intersection of curves of low degree.

The best bounds on the size of the second largest complete arc are the following.

For $q$ even Segre \cite{Segre1967} proved, combining Theorem~\ref{segrethm} with the Hasse-Weil theorem, that the second largest complete arc has size at most $q-\sqrt{q}+1$. The examples above imply that this bound is tight if $q$ is also a square.

For $q$ prime, Voloch \cite{Voloch1990b} proved, by using the St\"ohr-Voloch theorem from \cite{SV1986}, that the second largest complete arc has size at most $\frac{44}{45}q+\frac{8}{9}$.

For $q$ non-square, Voloch \cite{Voloch1991} proved the second largest complete arc has size at most $q-\frac{1}{4}\sqrt{pq} +\frac{29}{16}p-1$ if $q$ is odd and at most $q-\sqrt{2q} +2$ if $q$ is even. 

By, exploiting Theorem~\ref{segrethm} (Segre's theorem from \cite{Segre1967}) and bounds on the number of points on an algebraic curve, Hirschfeld and Korchm\'aros \cite{HK1996} proved that the second largest complete arc has size at most $q-\frac{1}{2} \sqrt{q}+5$, provided that the characteristic is at least $5$. This was subsequently improved to $q-\frac{1}{2} \sqrt{q}+3$ by Hirschfeld and Korchm\'aros \cite{HK1998}, provided that $q \geqslant 529$ and $q \neq 3^6, 5^5$.

Theorem~\ref{twocurves} leads directly to the following theorem, which improves on the previously known bounds.

\begin{theorem} \label{conicthmsquare}
An arc in ${\mathrm{PG}}(2,q)$,
 $q=p^{2h}$, $p\neq 2$, of size at least $q-\sqrt{q}+3+\sqrt{q}/p$ is contained in a conic. 
\end{theorem}

\begin{proof}
Let $A$ be a planar arc of size $q-\sqrt{q}+3+\sqrt{q}/p$, so $t=\sqrt{q}-\sqrt{q}/p-1$.

If $A$ is not contained in a curve of degree $t$ then Theorem~\ref{twocurves} implies that it is contained in two curves sharing no common component, each of degree at most $t+p^{\lfloor \log_p t \rfloor}$, which contain $A$. By Bezout's theorem, 
$$
|A| \leqslant (\sqrt{q}-1)^2,$$
a contradiction. Therefore, $A$ is contained in a conic, a contradiction.

If $A$ is contained in a curve of degree $t$ but not contained in a conic then consider $A'$ a subset of $A$ of size $q-\sqrt{q}+3=q+2-t'$. Since $t'=\sqrt{q}-1$, $p^{\lfloor \log_p t' \rfloor}=\sqrt{q}/p$ and
$$
p^{\lfloor \log_p t' \rfloor}(t'+\tfrac{1}{2}p^{\lfloor \log_p t' \rfloor}+\tfrac{3}{2}) \leqslant \tfrac{1}{2}(t'+2)(t'+1).
$$
By Theorem~\ref{twocurves} and Lemma~\ref{threepolys}, $A'$ is contained in the intersection of a curve of degree $\sqrt{q}-\sqrt{q}/p-1$ and a curve of degree $\sqrt{q}+\sqrt{q}/p-1$. Bezout's theorem implies
$$
|A'| \leqslant (\sqrt{q}-\sqrt{q}/p-1)(\sqrt{q}+\sqrt{q}/p-1),
$$
a contradiction. Hence, $A$ is contained in a conic.
\end{proof}

In the case that $q$ is a non-square and non-prime, Theorem~\ref{twocurves} does not improve upon the bound of Voloch mentioned above. However, in the case that $q$ is prime, it does improve on Voloch's bound for primes less than 1783.

\begin{theorem} \label{conicthmnonsquare}
An arc in ${\mathrm{PG}}(2,q)$, $q$ prime, of size at least $q-\sqrt{q}+\frac{7}{2}$ is contained in a conic. 
\end{theorem}

\begin{proof}
Let $A$ be a planar arc of size $\lceil q-\sqrt{q}+\frac{7}{2}\rceil$. Then $t=\lfloor \sqrt{q}-\frac{3}{2}\rfloor $ and $p^{\lfloor \log_p t \rfloor}=1$.
Since
$$
p^{\lfloor \log_p t \rfloor}(t+\tfrac{1}{2}p^{\lfloor \log_p t \rfloor}+\tfrac{3}{2}) \leqslant \tfrac{1}{2}(t+2)(t+1)
$$
Theorem~\ref{twocurves} implies that if $A$ is not contained in a conic then it is contained in two curves sharing no common component, each of degree at most $t+1$. By Bezout's theorem, 
$$
|A| \leqslant (\sqrt{q}-\tfrac{1}{2})^2,$$
a contradiction. Therefore, $A$ is contained in a conic.
\end{proof}



Our main motivation to start this work was to prove the following conjecture, which we call the $\sqrt{q}$ conjecture, see \cite[pp. 236]{Hirschfeld1998} and also \cite[Problem 1.1]{GPTU2002}.

\begin{conjecture} \label{rootq}
If $q\neq 9$ is an odd square  then the second largest complete arc in ${\mathrm{PG}}(2,q)$ has size $q-\sqrt{q}+1$.
\end{conjecture}

The examples of Kestenband imply that, if true, the bound in the conjecture is tight. 

Combining Segre~\cite{Segre1967}, Theorem~\ref{twocurves} and Theorem~\ref{conicthmsquare}, we get the following theorem.

\begin{theorem}
If there is a counterexample $A$ to Conjecture~\ref{rootq} of size $q+2-t$ then $\sqrt{q}-\sqrt{q}/p \leqslant t \leqslant \sqrt{q}$ and if $t \neq \sqrt{q}$, $A$ is contained in the intersection of two curves, sharing no common component, and each of degree at most $t+\sqrt{q}/p$.
\end{theorem}

\section{Bounds on arcs contained in low degree curves}

In 1973, Zirilli \cite{Zirilli1973} constructed arcs of size approximately $\frac{1}{2}q$, contained in a cubic curve by exploiting the group structure of an elliptic curve. If the group $H$ of the ${\mathbb F}_q$-rational points of a non-singular cubic curve has even order, then the coset of a subgroup of $H$ of index two is a planar arc. This was taken further by Voloch \cite{Voloch1987} who proved that Zirilli's construction leads to complete arcs of size $\frac{1}{2}(q+1)-\epsilon$, for all $|\epsilon|< \sqrt{q}$, for $q$ odd, $(\epsilon,q)=1$ and $q\geqslant 175$. Further results were obtained by Voloch \cite{Voloch1990a}, Szonyi \cite{Szonyi1987} and Giulietti \cite{Giulietti2002}.

Theorem~\ref{twocurves} leads to some surprising bounds if we assume that $A$ is contained in a curve $\psi$ of low degree $d$. To prove Theorem~\ref{boundeddeg}, we need a short lemma.

\begin{lemma} \label{threepolys}
Suppose $f$, $g$ and $h$ are three polynomials over ${\mathbb F}_q$ and that $h$ has degree less than $q$. If $f$ and $g$ are co-prime then there is a linear combination of $f$ and $g$ which is co-prime with $h$.
\end{lemma}

\begin{proof}
Suppose that for each $\lambda \in {\mathbb F}_q$ there is a factor $h_{\lambda}$ of $h$ that is a factor of $f+\lambda g$. Since $h$ has less than $q$ factors, there are $\lambda,\nu \in {\mathbb F}_q$ with $\lambda \neq \mu$ for which $h_{\lambda}=h_{\nu}$. But then $h_\lambda$ is a factor of both $f$ and $g$, contradicting the fact that $f$ and $g$ are co-prime.
\end{proof}

\begin{theorem} \label{boundeddeg}
If $A$ is a planar arc of $\mathrm{PG}(2,q)$ contained in a curve of degree $d<2p/(4+\sqrt{5})-1$, where $q=p^h$ is odd, then 
$$
|A| < \frac{d}{d+1}(q+1+\frac{q}{p})+1.
$$
\end{theorem}

\begin{proof}
Suppose that $|A|=q+2-t\geqslant (t+q/p)d+1$, i.e.
$$
t \leqslant \frac{p(q+1)-dq}{(d+1)p}.
$$

Let $c=\frac{1}{2}\sqrt{5}+1$. By hypothesis, $d<p/(c+1)-1$, which implies 
$$
q+2-t<(\frac{p}{c+1}-1)(t+\frac{q}{p})+1<\frac{pt+q}{c+1}+2-t
$$
and so
$t >cq/p$. 
Hence, $p^{\lfloor \log_p t \rfloor}=q/p$ and
$$
p^{\lfloor \log_p t \rfloor}(t+\tfrac{1}{2}p^{\lfloor \log_p t \rfloor}+\tfrac{3}{2})<\frac{t}{c}(t+\frac{t}{2c}+\tfrac{3}{2}) =(\frac{1}{c}+\frac{1}{2c^2})t^2+\frac{3t}{2c}<\tfrac{1}{2}(t+2)(t+1).
$$

By Theorem~\ref{twocurves}, there are two co-prime polynomials $f$ and $g$, of degree at most $t+\sqrt{q}/p$, whose zero-sets contain $A$. By Lemma~\ref{threepolys}, there is a linear combination of $f$ and $g$ which is co-prime to $h$, the polynomial whose zero set is $\psi$, the curve of degree $d$ containing $A$. 

Bezout's theorem implies that 
$$
|A|=q+2-t \leqslant (t+\frac{q}{p})d.
$$
a contradiction.

Therefore, 
$$
t> \frac{p(q+1)-dq}{(d+1)p},
$$
from which we get the desired bound.
\end{proof}

\section{Classification of arcs of size $q-1$ and $q-2$.}
It has been known for a long time that an arc of size $q$ is incomplete. This was proved by Segre \cite{Segre1955b} (with an amendment by B\"uke \cite{Buke1974}) for  $q$ odd, and by Tallini in \cite{Tallini1957} for $q$ even. The next question to answer (``Il primo nuovo quesito da porsi" as Segre writes in his 1955 paper) is whether there exists a complete arc of size $q-1$ in $\mathrm{PG}(2,q)$. Since the standard frame extended with the points $(1,2,3)$ and $(1,3,4)$ gives a complete arc of size 6 for $q=7$, the answer is affirmative. However, in spite of this small example, Segre expected the answer to be negative for sufficiently large $q$: ``Rimane tuttavia da indagare se, com'\`e da ritenersi probabile, la risposta alla suddetta domanda non diventi invece negativa quando si supponga $q$ sufficientemente grande". Segre's expectations were correct, there are no complete planar arcs of size $q-1$ for $q>13$. The classification of planar arcs of size $q-1$ is given in Corollary \ref{qminus1}.

Segre's bound $|S|\leqslant q-\sqrt{q}+1$ for $q$ even, proves the non-existence of complete arcs of size $q-1$ for $q$ even. 
For $q$ odd the situation is quite different. The combination of computer aided calculations (see for example the articles by Coolsaet and Sticker \cite{CS2009,CS2011} and Coolsaet \cite{Coolsaet2015}) with theoretical upper bounds (see Section \ref{section:rootq}) were not enough to complete the classification of complete arcs of size $q-1$.

According to \cite[Theorem 10.33]{Hirschfeld1998}, there are 15 values of $q$, all odd and satisfying $q\geqslant 31$ ($q=31$ has since been ruled out by Coolsaet \cite{Coolsaet2015}), for which it is not yet determined whether an arc of size $q-1$ is complete or not. We can now complete the classification of planar arcs of size $q-1$. 

\begin{corollary} \label{qminus1}
The only complete planar arcs of size $q-1$ occur for $q=7$ (there are $2$ projectively distinct arcs of size $6$), for $q=9$ (there is a unique arc of size $8$), $q=11$ (there is a unique arc of size $10$, see Example~\ref{10arc}) and $q=13$ (there is a unique arc of size $12$, see Example~\ref{12arc}). 
\end{corollary}

\begin{proof}
In \cite[Table 9.4]{Hirschfeld1998}, arcs of size $q-1$ are classified  for $q \leqslant 23$. 

Theorem~\ref{conicthmsquare} implies that an arc of size $q-1$ is contained in a conic for all $q$ square, $q \geqslant 25$. And Theorem~\ref{conicthmnonsquare}, that an arc of size $q-1$ is contained in a conic for all $q$ prime, $q \geqslant 23$.

\end{proof}

We can also complete the classification of planar arcs of size $q-2$. 

\begin{corollary} \label{qminus2}
The only complete planar arcs of size $q-2$ occur for $q=8$ (there are $3$ projectively distinct arcs of size $6$), for $q=9$ (there is a unique arc of size $7$) and $q=11$ (there are $3$ projectively distinct arcs of size $9$). 
\end{corollary}

\begin{proof}
In \cite[Table 9.4]{Hirschfeld1998}, arcs of size $q-2$ are classified for $q \leqslant 23$. 

Theorem~\ref{conicthmsquare} implies that an arc of size $q-2$ is contained in a conic for all $q$ square, $q \geqslant 49$. And Theorem~\ref{conicthmnonsquare}, that an arc of size $q-2$ is contained in a conic for all $q$ prime, $q \geqslant 31$.

The remaining cases are ruled out by the afore-mentioned computational results of Coolsaet and Sticker \cite{CS2009,CS2011} and Coolsaet \cite{Coolsaet2015}.
\end{proof}

As Segre would say ``Il primo nuovo quesito da porsi" is whether there are any more arcs of size $q-3$ to be found. The computational results of Coolsaet and Sticker \cite{CS2009,CS2011} and Coolsaet \cite{Coolsaet2015} rule out any examples being found for $19 \leqslant q \leqslant 31$.  Theorem~\ref{twocurves} and Bezout's theorem rule out any new examples being found for $q \geqslant 41$. Voloch's bound from \cite{Voloch1991} rules out $q=32$. We can deduce from \cite[Table 9.4]{Hirschfeld1998} that there are examples of complete arcs of size $q-3$ for $q=9,13,16$ and $17$. The only remaining case is now $q=37$. Furthermore, Theorem~\ref{twocurves} implies that if there is an arc of size $34$ in $\mathrm{PG}(2,37)$ then it is contained in the intersection of two sextic curves, which do not share a common component.

\section{The tangent functions.}

Let $S$ be an arc of $\mathrm{PG}(2,q)$ of size $q+2-t$.

Each point of $S$ is incident with precisely $t$ tangents, where a tangent is a line of 
$\mathrm{PG}(2,q)$ which is incident with exactly one point of $S$.

We define a homogeneous polynomial of degree $t$ in three variables $X=(X_1,X_2,X_3)$
$$
f_a(X)=\prod_{i=1}^t \alpha_i(X),
$$
where $\alpha_i(X)$ are linear forms whose kernels are the $t$ tangents to the arc $S$ incident with a point $a$ of $S$. This defines $f_a(X)$ up to scalar factor. Let $e$ be an arbitrary fixed element of $S$ and scale $f_a(X)$ so that 
$$
f_a(e)=(-1)^{t+1}f_{e}(a).
$$

The following is Segre's lemma of tangents, originally from \cite{Segre1967}. It appears again in a coordinate-free version in \cite{Ball2012}, and here it is further simplified by scaling.

\begin{lemma} \label{segre}
For all $x,y \in S$, 
$$
f_x(y)=(-1)^{t+1} f_y(x).
$$
\end{lemma}

\begin{proof} 
Let $B=\{e,x,y\}$. Since $S$ is an arc, it follows that $B$ is a basis. 

For $a \in S$, let $f_a^*(X)$ be the polynomial we obtain from $f_a(X)$ when we change the basis from the canonical basis to the basis $B$. Then
$$
f_x^*(X)=\prod_{i=1}^t (b_{i1}X_1+b_{i3}X_3),\ \
f_y^*(X)=\prod_{i=1}^t (c_{i1}X_1+c_{i2}X_2),\ \
\mathrm{and} \ \
f_{e}^*(X)=\prod_{i=1}^t (d_{i2}X_2+d_{i3}X_3),
$$ 
for some $b_{ij}, c_{ij}, d_{ij} \in {\mathbb F}_q$.

With respect to the basis $B$, the line joining $e$ and $(s_1,s_2,s_3) \in S \setminus B$ is the kernel of the linear form $X_2-(s_2/s_3)X_3$. Since these lines are distinct from the tangent lines incident with $e$ we have
$$
\prod_{i=1}^t \frac{d_{i3}}{d_{i2}} \prod_{s \in S \setminus B} \frac{-s_2}{s_3}=\prod_{b \in {\mathbb F}_q \setminus \{ 0\}} b = -1.
$$
Observing that $\prod_{i=1}^t d_{i3}=f_{e}^*(y)$ and $\prod_{i=1}^t d_{i2}=f_{e}^*(x)$, this gives
$$
f_{e}^*(y)\prod_{s \in S \setminus B} s_2=(-1)^{t+1}f_{e}^*(x)\prod_{s \in S \setminus B} s_3.
$$
Similarly,
$$
f_{x}^*(e)\prod_{s \in S \setminus B} s_3=(-1)^{t+1}f_{x}^*(y)\prod_{s \in S \setminus B} s_1
$$
and
$$
f_{y}^*(x)\prod_{s \in S \setminus B} s_1=(-1)^{t+1}f_{y}^*(e)\prod_{s \in S \setminus B} s_2.
$$
Combining these three equations we obtain,
$$
f_{e}^*(y)f_{x}^*(e)f_{y}^*(x)=(-1)^{t+1}f_{e}^*(x)f_{x}^*(y)f_{y}^*(e).
$$
Since $f_a$ and $f_a^*$ define the same function on the points of $\mathrm{PG}(2,q)$, we also have that $f_x^*(e)=(-1)^{t+1}f_{e}^*(x)$ and $f_y^*(e)=(-1)^{t+1}f_{e}^*(y)$, and hence $f_x^*(y)=(-1)^{t+1}f_y^*(x)$, and therefore $f_x(y)=(-1)^{t+1}f_y(x)$.
\end{proof}

In 1967 Segre \cite{Segre1967} proved that the set of tangents, viewed as points in the dual plane, lie on an algebraic curve of degree $t$, when $q$ is even, and an algebraic curve of degree $2t$, when $q$ is odd. His main tool in the proof of this theorem was Lemma~\ref{segre}. However, to prove this result for $q$ odd, he only uses the fact that
$$
f_x(y)^2=f_y(x)^2,
$$
whereas Lemma~\ref{segre} maintains that the stronger statement 
$$
f_x(y)=(-1)^{t+1}f_y(x)
$$
holds.

In terms of the polynomials $f_a(X)$, Segre's theorem is the following theorem. Although we will not need this to derive the main results of this article we include a proof here, since it is quite short.

\begin{theorem} \label{segrethm}
Let $m\in \{1,2 \}$ such that $m-1=q$ modulo $2$. If $S$ is a planar arc of size $q+2-t$, where $|S| \geqslant mt+2$, then there is a homogeneous polynomial in three variables $\phi(Z)$, of degree $mt$, which gives a polynomial $G(X,Y)$ under the substitution $Z_1=X_2Y_3-Y_2X_3$, $Z_2=X_1Y_3-Y_1X_3$, $Z_3=X_2Y_1-Y_2X_1$, with the property that for all $a\in S$
$$
G(X,a)=f_a(X)^m.
$$ 
\end{theorem}

\begin{proof} 
Order the set $S$ arbitrarily and let $E$ be a subset of $S$ of size $mt+2$. Define
$$
G(X,Y)=\sum_{a<b} f_{a}(b)^m \prod_{u \in E \setminus \{a,b\}} \frac{\det(X,Y,u)}{\det(a,b,u)}.
$$
where the sum runs over subsets $\{a,b\}$ of $E$.

Observe that  $\det(X,Y,u)=u_1Z_1+u_2Z_2+u_3Z_3$, where $Z_1=X_2Y_3-Y_2X_3$, $Z_2=X_3Y_1-Y_3X_1$ and $Z_3=X_1Y_2-X_2Y_1$, so $G$ can be obtained from a homogeneous polynomial in $Z$ of degree $mt$ under this change of variables.

Then, for $y \in E$, the only non-zero terms in $G(X,y)$ are obtained for $a=y$ and $b=y$. Lemma~\ref{segre} implies
$$
G(X,y)=\sum_{a\in E \setminus {y}} f_{a}(y)^m \prod_{u \in E \setminus \{a,y\}}\frac{\det(X,y,u)}{\det(a,y,u)}.
$$
With respect to a basis containing $y$, the polynomials $G(X,y)$ and $f_y(X)^m$ are homogeneous polynomials in two variables of degree $mt$. Their values at the $mt+1$ points $x \in E \setminus \{y\}$ are the same, so we conclude that $G(X,y)=f_y(X)^m$.

If $y \not\in E$ then we still have that with respect to a basis containing $y$, the polynomial $G(X,y)$ is a homogeneous polynomial in two variables of degree $mt$. For $x \in E$,
$$
G(x,y)=G(y,x)=f_x(y)^m=f_y(x)^m,
$$
the last equality following from Lemma~\ref{segre}, and so again we conclude that $G(X,y)=f_y(X)^m$.
\end{proof}

\section{A tangent polynomial for the arc}

Let $V_r[X]$ denote the vector space of homogeneous polynomials of degree $r$ in ${\mathbb{F}}_q[X_1,X_2,X_3]$.

For a set $D$ of points of $\mathrm{PG}(2,q)$, we define an {\em $r$-socle} of $D$ as follows.
Let $\mathrm{M}_r(D)$ denotes the 
matrix  whose rows are indexed by the elements of $V_r[X]$ and whose columns are indexed by the points of $D$, and where the $(f(X),x)$-entry is $f(x)$. Then an {\em $r$-socle} of $D$ is a subset of 
$D$ whose elements index the columns which form a basis for the column space of $\mathrm{M}_r(D)$.

Let $A$ be an {\em arc} of $\mathrm{PG}(2,q)$ of size $q+2-t$, and let $\Phi_r$ denote the subspace of polynomials of $V_r[X]$ which are zero on $A$.

\begin{lemma} \label{rsocle}
An $r$-socle of $A$ has size ${r+2 \choose 2}-\dim \Phi_r$.
\end{lemma}

\begin{proof}
By definition, the size of an $r$-socle of $A$ is equal to the rank of the matrix $\mathrm{M}_r(A)$, which is ${t+2 \choose 2} -\dim \Phi_r$. 
\end{proof}

\begin{lemma} \label{tplusrsocle}
Let $S_0$ be an $t$-socle of $A$. For all positive integers $r$, there is a $(t+r)$-socle $S_r$ of $A$ containing a $(t+j)$-socle $S_j$, for all $j=0,\ldots,r-1$. Moreover, 
$$
|S_r\setminus S_0| \leqslant r(t+\tfrac{1}{2}r+\tfrac{3}{2}).
$$
\end{lemma}

\begin{proof}
If $S_j$ is a $(t+j)$-socle of $A$, then the
elements of $S_j$ index a basis for the column space of $\mathrm{M}_{t+j}(A)$. 
Moreover, the set of columns of $\mathrm{M}_{t+j+1}(A)$ indexed by $S_j$ are linearly independent
vectors in the column space of $\mathrm{M}_{t+j+1}(A)$. Hence, $S_j$ can be extended to a $(t+j+1)$-socle $S_{j+1}$.

If $B$ is a basis for $\Phi_r$ then $\{X_1^r \phi \ | \ \phi \in B\}$ are linearly independent polynomials in $\Phi_{t+r}$. Hence, 
$$
\dim \Phi_{t+r} \geqslant \dim\Phi_t,
$$
which, by Lemma~\ref{rsocle}, implies
$$
{t+r+2 \choose 2} - |S_r| \geqslant {t+2 \choose 2} - |S_0|.
$$
\end{proof}

Note that in the following theorem, Theorem~\ref{main}, if $A$ is not contained in a curve of degree $t$ then $\Phi_t=\{0\}$ and so we conclude that there is a polynomial $F(X,Y)$ with the property that 
$$
F(X,y)=f_y(X),
$$
for all $y \in A$.

A polynomial in ${\mathbb{F}}_q[X,Y]$ is called a {\em $(t,t)$-form} if it is simultaneously homogeneous of degree $t$ in both sets of variables
$X=(X_1,X_2,X_3)$ and $Y=(Y_1,Y_2,Y_3)$.

\begin{theorem} \label{main}
{\rm
Let $A$ be an arc of size $q+2-t$ of $\mathrm{PG}(2,q)$. For any subset $S$ of $A$ containing a $t$-socle $S_0$ of $A$, where 
$$
|S \setminus S_0|= \min\{ \tfrac{1}{2}(t+2)(t+1),|A \setminus S_0| \},
$$
there is a $(t,t)$-form $F(X,Y)\in {\mathbb{F}}_q[X,Y]$ such that
$$
F(X,y)=f_y(X) \pmod {\Phi_t},
$$
for all $y \in A$ and 
$$
F(X,y)=f_y(X),
$$
for all $ y\in S$.}
\end{theorem}

\begin{proof}

Let $F(X,Y)$ be a $(t,t)$-form $F(X,Y)\in {\mathbb{F}}_q[X,Y]$ whose ${t +2 \choose 2}^2$ coefficients will be determined as the solution of a homogeneous system of equations.

For a point $y$ of $A$, let $\mathrm{Tan}(y)$ be the set of points on the tangents to $A$ incident with $y$.

Let $e$ be a point of $S_0$.

Let $T(e)$ be a $t$-socle for $\mathrm{Tan}(e)$. Then $T(e)$ is a set of ${t+2 \choose 2}-1$ points, since $V(f_e)$ is the unique curve of degree $t$ containing $\mathrm{Tan}(e)$. 

Impose the 
\begin{eqnarray}\label{eqn:conditions1}
{t+2 \choose 2}-1
\end{eqnarray}
conditions $F(z,e)=0$ for all $z \in T(e)$. Then $F(X,e)$ is a multiple of $f_e(X)$. Now scale $F(X,Y)$ so that
$$
F(X,e)=f_e(X).
$$
For each $y \in S_0 \setminus \{ e \}$, let $T(y)$ be a $t$-socle of $\mathrm{Tan}(y)$, which again is set of ${t+2 \choose 2}-1$ points.

For $y \in S_0 \setminus \{ e \}$, impose the condition $F(z,y)=0$ for all $z \in T(y)$ and the condition 
$$F(y,e)-(-1)^{t+1}F(e,y)=0.$$ This amounts to
\begin{eqnarray}\label{eqn:conditions2}
(|S_0|-1){t+2 \choose 2}=({t+2 \choose 2}-\dim\Phi_t-1){t+2 \choose 2}
\end{eqnarray}
conditions. 
Then, since
$$
F(e,y)=(-1)^{t+1}F(y,e)=(-1)^{t+1}f_e(y)=f_y(e),
$$
the homogeneous polynomial
$$F(X,y)-f_y(X)
$$
defines a curve of degree $t$ which is zero at all the points of $T(y)\cup \{ e\}$. By the definition of $T(y)$ it is a multiple of $f_y(X)$, but since $f_y(e) \neq 0$, it is zero. Hence
$$
F(X,y)=f_y(X),
$$
for all $y \in S_0$. This implies that for all $x,y\in S_0$, 
$$
F(x,y)=f_y(x)=(-1)^{t+1}f_x(y)=(-1)^{t+1}F(y,x).
$$
Then for each $y\in S_0$, the  curve of degree $t$ defined by
$$
F(X,y)-(-1)^{t+1}F(y,X)
$$
contains $S_0$. Since $S_0$ is a $t$-socle of $A$, it contains $A$, i.e.
$$
F(x,y)=(-1)^{t+1}F(y,x),
$$
for all $(x,y) \in (S_0\times A) \cup (A\times S_0)$.

This implies that, for $y \in A$ and $x \in S_0$,
$$
F(y,x)=(-1)^{t+1}F(x,y)=(-1)^{t+1}f_x(y)=f_y(x).
$$
Now consider 
$$
F(X,y)-(-1)^{t+1}F(y,X)
$$
for $y \in A$. Again, this defines a curve of degree $t$ which contains $S_0$ and so contains $A$. 
Thus,
$$
F(x,y)-(-1)^{t+1}F(y,x)=0,
$$
for all $x,y\in A$.

Since for $y \in A$ and $x \in S_0$
$$
F(x,y)=(-1)^{t+1}F(y,x)=(-1)^{t+1}f_x(y)=f_y(x),
$$
the polynomial
$$
F(X,y)-f_y(X)
$$
defines a curve of degree $t$ which is zero on $S_0$, and therefore zero on $A$, i.e. it belongs to $\Phi_t$. Hence, 
$$
F(X,y)=f_y(X) \pmod{\Phi_t}.
$$

For $y \in S \setminus S_0$, let $T(y)$ be a subset of $\mathrm{Tan}(y)$ such that $S_0 \cup T(y)$ is  $t$-socle for $\mathrm{PG}(2,q)$. Such a set exists since the only curve of degree $t$ which contains $\mathrm{Tan}(y)$ is $V(f_y)$, which does not contain any of the points in $S_0$. A $t$-socle for $\mathrm{PG}(2,q)$ has ${t+2 \choose 2}$ points, so $|T(y)|=\dim \Phi_t$.

For each $y \in S \setminus S_0$ and $z\in T(y)$, impose the condition $F(z,y)=0$. This amounts to 
\begin{eqnarray}\label{eqn:conditions3}
|S\setminus S_0|\dim\Phi_t\leq {t+2 \choose 2}\dim \Phi_t
\end{eqnarray}
further conditions. Adding the number of conditions in (\ref{eqn:conditions1}), (\ref{eqn:conditions2}), (\ref{eqn:conditions3}), we have
now imposed at most 
$$
{t+2 \choose 2}^2-1
$$
conditions. The curve of degree $t$
$$
F(X,y)-f_y(X)
$$
is zero at all the points of $T(y)\cup S_0$, so it is zero. Hence, for all $y\in S$,
$$
F(X,y)=f_y(X),
$$
which completes the proof.
\end{proof}


\section{Some examples.}


\begin{example} \label{tis1}
\rm{
If $t=1$ then clearly $\Phi_t=\{0\}$ and Theorem~\ref{main} implies that the polynomial $F(X,Y)$ is a symmetric bilinear form.

If $F(X,X)=0$ then $F(X,Y)$ is also an alternating bilinear form, so $q$ is even in this case. The polynomial $F(X,Y)$ is equal to $\det(X,Y,z)$ for some point $z$, which is incident with all the tangents and is called the {\em nucleus}. The arc $S$ can be extended to an arc of size $q+2$ by adjoining the nucleus.

If $F(X,X) \neq  0$ then Theorem~\ref{main} implies that the points of $S$ are the zeros of a conic, since $f_x(x)=0$ for all $x \in S$. Thus we have that if $q$ is odd then an arc of size $q+1$ is a conic, which is Segre's theorem from \cite{Segre1955a}.
}
\end{example}

\begin{example} \label{12arc}
\rm{The planar arc of 12 points in $\mathrm{PG}(2,13)$,
$$
A=\{ ( 3,4,1),(-3,4,1),(3,-4,1),(-3,-4,1),(4,3,1),(4,-3,1),(-4,3,1),(-4,-3,1),
$$
$$
(1,1,1),(1,-1,1),(-1,1,1),(-1,-1,1) \}
$$
is an arc with $t=3$ and it is not contained in a curve of degree $3$. Consequently, Theorem~\ref{main} implies that there is a $(3,3)$-form $F(X,Y)$ with the property that $F(X,a)=f_a(X)$ for all $a \in A$. It is given by
$$
F(x,y)=5(x_2^2x_3y_1^2y_3+y_2^2y_3x_1^2x_3+x_2x_3^2y_1^2y_2+x_1^2x_2y_2y_3^2+x_1x_3^2y_1y_2^2+x_1x_2^2y_1y_3^2)$$
$$
+6 x_1x_2x_3y_1y_2y_3+x_1^3y_1^3+x_2^3y_2^3+x_3^3y_3^3.
$$
Theorem~\ref{twocurves} implies that this example is contained in the intersection of two quartic curves which do not share a common component. As observed in \cite{AHK1994}, $A$ lies on the intersection of the three quartic curves $x_3^4=x_1^2x_2^2$, $x_2^4=x_1^2x_3^2$ and $x_1^4=x_3^2x_2^2$, .}
\end{example}

\begin{example} \label{10arc}
\rm{The planar arc of 10 points in $\mathrm{PG}(2,11)$,
$$
A=\{ ( 1, 0, 0 ), ( 0, 1, 0 ), 
  ( 0, 0, 1), (1, 8, 7 ), 
  ( 1, 5, 2 ), ( 1, 3, 8 ), 
  ( 1, 4, 1 ), ( 1, 7, 6 ), 
  (1, 9, 4 ), ( 1, 10, 5 )\}
 $$
is an arc with $t=3$ and it is not contained in a curve of degree $3$. Consequently, Theorem~\ref{main} implies that there is a $(3,3)$-form $F(X,Y)$ with the property that $F(X,a)=f_a(X)$ for all $a \in A$. In this case $F$ has 59 terms and is given by
 $$
F(x,y)=
9   x_3^3  y_2^3 
+7   x_3^3 y_1 y_2^2 
+4   x_3^3 y_1^2 y_2 
+2   x_3^3 y_1^3  
+5  x_2 x_3^2  y_2 y_3^2
+4  x_2 x_3^2 y_1  y_3^2
+ x_2 x_3^2 y_1 y_2 y_3
$$
$$
+9  x_2 x_3^2 y_1 y_2^2 
+2  x_2 x_3^2 y_1^2  y_3
+8  x_2 x_3^2 y_1^2 y_2 
+5  x_2^2 x_3  y_2^2 y_3
+9  x_2^2 x_3 y_1  y_3^2
+4  x_2^2 x_3 y_1 y_2 y_3
+8  x_2^2 x_3 y_1 y_2^2 
$$
$$
+5  x_2^2 x_3 y_1^2  y_3
+ x_2^2 x_3 y_1^2 y_2 
+8  x_2^2 x_3 y_1^3  
+9  x_2^3    y_3^3
+3  x_2^3  y_1^2  y_3
+  x_2^3  y_1^3  
+4 x_1  x_3^2  y_2 y_3^2
+9 x_1  x_3^2  y_2^2 y_3
$$
$$
+5 x_1  x_3^2 y_1  y_3^2
+10 x_1  x_3^2 y_1 y_2 y_3
+8 x_1  x_3^2 y_1 y_2^2 
+9 x_1  x_3^2 y_1^2 y_2 
+ x_1 x_2 x_3  y_2 y_3^2
+4 x_1 x_2 x_3  y_2^2 y_3
+10 x_1 x_2 x_3 y_1  y_3^2
$$
$$
+7 x_1 x_2 x_3 y_1 y_2 y_3
+8 x_1 x_2 x_3 y_1 y_2^2 
+4 x_1 x_2 x_3 y_1^2  y_3
+3 x_1 x_2 x_3 y_1^2 y_2 
+7 x_1 x_2^2    y_3^3
+9 x_1 x_2^2   y_2 y_3^2
$$
$$
+8 x_1 x_2^2   y_2^2 y_3
+8 x_1 x_2^2  y_1  y_3^2
+8 x_1 x_2^2  y_1 y_2 y_3
+8 x_1 x_2^2  y_1 y_2^2 
+10 x_1 x_2^2  y_1^2  y_3
+ x_1 x_2^2  y_1^2 y_2 
$$
$$
+2 x_1^2  x_3  y_2 y_3^2
+5 x_1^2  x_3  y_2^2 y_3
+3 x_1^2  x_3  y_2^3 
+4 x_1^2  x_3 y_1 y_2 y_3
+10 x_1^2  x_3 y_1 y_2^2 
+5 x_1^2  x_3 y_1^2  y_3
$$
$$
+3 x_1^2  x_3 y_1^2 y_2 
+4 x_1^2 x_2    y_3^3
+8 x_1^2 x_2   y_2 y_3^2
+ x_1^2 x_2   y_2^2 y_3
+9 x_1^2 x_2  y_1  y_3^2
+3 x_1^2 x_2  y_1 y_2 y_3
$$
$$
+ x_1^2 x_2  y_1 y_2^2 
+3 x_1^2 x_2  y_1^2  y_3
+8 x_1^2 x_2  y_1^2 y_2 
+2 x_1^3     y_3^3
+8 x_1^3    y_2^2 y_3
+ x_1^3    y_2^3.
$$
}
\end{example}

\begin{example} \label{24arc}
\rm{The $24$ points which form the zero set of the polynomial
$$
h(x)=x_1^3x_2+x_2^3x_3+x_3^3x_1
$$
is a planar arc of $\mathrm{PG}(2,29)$ with $t=7$. The dimension of the space of homogeneous polynomials of degree $t$ whose zero set contains the 24 points of the arc is $13$. The subspace of this space which are polynomials which are multiples of $h$ has dimension $10$, so there is a polynomial of degree $7$ which is zero on the arc and is co-prime to $h$. Observe that Theorem~\ref{twocurves}, and the argument in the proof of Theorem~\ref{boundeddeg}, ensure that there is a polynomial of degree $8$ which is zero on the arc and is co-prime to $h$. Theorem~\ref{boundeddeg} implies that an arc of $\mathrm{PG}(2,29)$ contained in a quartic curve has at most $25$ points.


}
\end{example}

\section{Arcs in odd characteristic are contained in low degree curves.}

In this section we investigate some consequences of Theorem~\ref{main}.

The following lemma will be used in the proof of Lemma~\ref{almost}. It is a rather technical lemma and can be skipped on first reading.

\begin{lemma} \label{tech}
Let $t$ be a positive integer and $\epsilon=\lfloor \log_p t \rfloor$. Suppose that for all $k \in \{1,\ldots,p^{\epsilon}\}$ and for all $i \in \{0,\ldots,t-k\}$,
$$
\sum_{d=i}^{i+k} {d \choose i}{t-d \choose t-i-k} g_d=0, 
$$
for some $g_0,\ldots,g_t$ in a polynomial ring over ${\mathbb F}_q$.
Then $$g_t+(-1)^{t+1}g_0=0.$$
\end{lemma}

\begin{proof}
Let $t=t_0+t_1p+\cdots+t_{\epsilon}p^{\epsilon}$ be the $p$-ary expansion of $t$.

First, we prove by induction that 
$$
{t-i \choose s} g_i+ (-1)^{s+1} {i+s \choose i} g_{i+s}=0,
$$
for all $s \in \{ 1, \ldots, p^{\epsilon} \}$ and for all $i \in \{0,\ldots, t-s\}$.

For $s=1$ this is the given equation with $k=1$. 

Assume the equation holds for $j \leqslant s-1$. Then
$$
{i+s \choose i}g_{i+s}=-\sum_{d=i}^{i+s-1} {d \choose i}{t-d \choose t-i-s} g_d=-\sum_{j=0}^{s-1} {i+j \choose i}{t-i-j \choose t-i-s}g_{i+j},
$$
which, by induction, is equal to
$$
-\sum_{j=0}^{s-1} (-1)^j {t-i \choose j}{t-i-j \choose t-i-s}g_{i}=-\sum_{j=0}^{s-1} (-1)^j {s \choose j}{t-i \choose s}g_i= (-1)^s{t-i \choose s}g_i,
$$
as required.

For $s=p^{\epsilon}$ and $i=t-(k+1)p^{\epsilon}$, where $k=0,\ldots,t_{\epsilon}-1$, this gives
$$
(k+1)g_{t-(k+1)p^{\epsilon}}+(t_{\epsilon}-k)g_{t-kp^{\epsilon}}=0,
$$
by Lucas' theorem.

Combining these equations we have that
$$
g_{t-t_{\epsilon}p^{\epsilon}}=(-1)^{t_{\epsilon}}g_t.
$$

For $s=t-t_{\epsilon}p^{\epsilon}$ and $i=0$ we obtain
$$
g_0+(-1)^{t-t_{\epsilon}p^{\epsilon}+1}g_{t-t_{\epsilon}p^{\epsilon}}=0.
$$
Eliminating $g_{t-t_{\epsilon}p^{\epsilon}}$ from these two equations, the lemma follows.
\end{proof}

We say that a polynomial $\phi\in {\mathbb{F}}_q[Z_1,X_2,X_3]$ is {\em hyperbolic on an arc $A$} in ${\mathrm{PG}}(2,q)$, if it has the property that if the kernel of a linear form $\gamma$ is a bisecant to $A$ then $\phi$ modulo $\gamma$ factorises into at most two linear factors whose multiplicities sum to the degree of $\phi$ and which are zero at the points of $A$ on the bisecant.

\begin{lemma} \label{almost}
Let $A$ be a planar arc of size $q+2-t$, $q$ odd,  and let $F(X,Y)$ be a $(t,t)$ form given by Theorem~\ref{main}.
If $A$ has a subset $S$ which contains a $(t+r)$-socle, for all $r=0,\ldots,p^{\lfloor \log_p t \rfloor}$, for which
$$
F(X,y)=f_y(X)
$$
for all $y \in S$ or, there is no curve of degree $t$ containing $A$, in which case $S=A$, then at least one of the following holds:
\begin{enumerate}
\item[(i)] there is a homogeneous polynomial of degree at most $t+p^{\lfloor \log_p t \rfloor}$ which is hyperbolic on $S$;
\item[(ii)] there are two co-prime homogeneous polynomials of degree at most $t+p^{\lfloor \log_p t \rfloor}$ which are zero on $A$.
\end{enumerate}
\end{lemma}

\begin{proof}

For each $v=(i,j,k) \in \{0,\ldots,t-1\}^3$ where $i+j+k \leqslant t-1$, define $\rho_{v}(Y)$
to be the coefficient of $X_1^iX_2^jX_3^k$ in 
$$
F(X+Y,Y)-F(X,Y).
$$
Observe that the degree of $\rho_{v}(Y)$ is $2t-i-j-k$.

Since 
$$
F(X,y)=f_y(X)=f_y(X+y)=F(X+y,y)
$$
for all $y \in S$, we have that $\rho_{w}(y)=0$ for all $w \in W$ and $y \in S$, where
$$
W=\{(i,j,k) \in \{0,\ldots,t-1\}^3 \ | \ t-p^{\lfloor \log_p t \rfloor} \leqslant i+j+k \leqslant t-1\}.
$$
This implies that $\rho_{w}(y)=0$ for all $w \in W$ and $y \in A$, since in the case that $S \neq A$, the set $S$ contains a $(t+j)$-socle, for all $j=0,\ldots,p^{\lfloor \log_p t \rfloor}$, which implies $\rho_{w}(y)=0$ for all $y \in A$ too.

Let $\phi(Y)$ be the greatest common divisor of
$$
\langle \{ \rho_w(Y) \ | \ w \in W\}  \cup \Phi_t[Y] \rangle,
$$
where $\Phi_t[Y]$ denotes the subspace of homogeneous polynomials of degree $t$ in ${\mathbb{F}}_q[Y_1,Y_2,Y_3]$ which are zero on $A$.
Observe that the degree of $\phi$ is at most $t+p^{\lfloor \log_p t \rfloor}$. In the case that $A$ is not contained in a curve of degree $t$, we do not yet discount the case that this subspace is $\{0\}$. In this case, which we shall rule out, $\phi$ is the zero polynomial.

Let $x$ and $y$ be arbitrary points of $S$. Let $B$ be a basis, with respect to which, $x=(1,0,0)$ and $y=(0,1,0)$. Let $f_a^*(X)$ be the polynomial we obtain from $f_a(X)$ when we change the basis from the canonical basis to $B$, and likewise let $F^*(X,Y)$ be the polynomial we obtain from $F(X,Y)$, and let $\phi^*$ be the polynomial we get from $\phi$.

Define homogeneous polynomials $b_{d_1d_2d_3}(Y)$ of degree $t$ by 
writing
 $$
F^*(X,Y)=\sum_{d_1+d_2+d_3=  t} b_{d_1d_2d_3}(Y) X_1^{d_1}X_2^{d_2} X_3^{d_3}.
 $$
Then
 $$
F^*(X+Y,Y)=\sum_{d_1+d_2+d_3=  t} b_{d_1d_2d_3}(Y) (X_1+Y_1)^{d_1}(X_2+Y_2)^{d_2} (X_3+Y_3)^{d_3},
 $$
 $$
 =\sum_{d_1+d_2+d_3=  t} b_{d_1d_2d_3}(Y) {d_1 \choose i_1}  {d_2 \choose i_2} {d_3 \choose i_3}X_1^{i_1} Y_1^{d_1-i_1}X_2^{i_2} Y_2^{d_2-i_2}X_3^{i_3} Y_3^{d_3-i_3}.
 $$
Let $r_{ijk}(Y)$ be the coefficient of $X_1^{i}X_2^jX_3^k$ in 
$$
F^*(X+Y,Y)-F^*(X,Y).$$
Then $r_{ijk}(Y)$ is in
$$
\langle \rho^*_w(Y) \ | \ w_1+w_2+w_3=i+j+k\rangle,$$
where $\rho^*_w(Y)$ is the polynomial we obtain from $\rho_w(Y)$, when we change the basis from the canonical basis to $B$. 

Since
$\phi$ is the greatest common divisor of
$$
\langle \{\rho_w(Y) \ | \ w \in W \}\cup \Phi_t[Y] \rangle,
$$
$\phi^*$ is a factor of all the polynomials in
$$
\langle \{r_w(Y) \ | \ w \in W \} \cup \Phi_t^*[Y] \rangle,
$$
where $\Phi_t^*[Y]$ is the subspace of homogeneous polynomials of degree $t$ which are zero on $S$, with respect to the basis $B$.

Let $w=(i,t-i-k,0)$, where $k \in \{1,\ldots,p^{\epsilon} \}$ and $i \in \{0,\ldots,t-k\}$. Then $w \in W$ and 
$$
r_w(Y)=\sum_{d_1+d_2+d_3=t} {d_1 \choose i}{d_2 \choose t-i-k} Y_1^{d_1-i} Y_2^{d_2-t+i+k} Y_3^{d_3} b_{d_1d_2d_3}(Y).
$$
The polynomial $\phi^*$ is a factor of all these polynomials, so it is a factor of $Y_1^i Y_2^{t-i-k} r_w(Y)$ and therefore,
$$
\sum_{d=i}^{i+k} {d \choose i}{t-d \choose t-i-k} Y_1^dY_2^{t-d} b_{d,t-d,0}(Y)=0 \pmod {Y_3,\phi^*},
$$
for all $k \in \{1,\ldots,p^{\epsilon} \}$ and $i \in \{0,\ldots,t-k\}$. 

By Lemma~\ref{tech}, with $g_d(Y)=Y_1^dY_2^{t-d} b_{d,t-d,0}(Y)$,
$$
Y_1^t b_{t,0,0}(Y)+(-1)^{t+1}Y_2^t b_{0,t,0}(Y)=0 \pmod {Y_3,\phi^*}.
$$
Note that if $\phi^*=0$ then $\rho_w=r_w=0$ for all $w\in W$, so the expression is also zero in this case.

By Theorem~\ref{main},
$$
F(Y,x)=f_{x}(Y).
$$
With respect to the basis $B$ this gives,
$$
F^*(Y,x)=f^*_x(Y).
$$
Since $f_{x}^*(Y)$ is a polynomial in $Y_2$ and $Y_3$,
$$
f_{x}^*(Y)=f_{x}^*(y)Y_2^t \pmod{ Y_3}.
$$

Again, by Theorem~\ref{main}, the polynomial
$$
F(Y,x)-(-1)^{t+1}F(x,Y)
$$
is a polynomial of degree $t$ which is zero on $A$. 

Therefore, it is a polynomial of $\Phi_t[Y]$ and so
$$
F(Y,x)-(-1)^{t+1}F(x,Y)=0 \pmod{\phi}.
$$
Hence,
$$
F^*(Y,x)=(-1)^{t+1}F^*(x,Y) \pmod{\phi^*}.
$$

This implies that 
$$
b_{t,0,0}(Y)=F^*(x,Y)=(-1)^{t+1}F^*(Y,x)=(-1)^{t+1}f_{x}^*(y)Y_2^t \pmod{\phi^*, Y_3}.
$$
Similarly
$$
b_{0,t,0}(Y)=(-1)^{t+1}f_{y}^*(x)Y_1^t \pmod{\phi^*,Y_3}.
$$
Hence, we have that 
$$
Y_1^tY_2^t(f_{x}^*(y)+(-1)^{t+1}f_{y}^*(x))=0\pmod {\phi^*,Y_3}.
$$
By Lemma~\ref{segre} and the fact that $f_a$ and $f_a^*$ define the same functions, this implies
$$
2Y_1^tY_2^t f_{x}(y)=0 \pmod {\phi^*,Y_3}.
$$
By hypothesis $q$ is odd, so the left-hand side is non-zero. Hence, this equation rules out the possibility that $\phi^*$ (and hence $\phi$) is zero and so we have proved that there is a curve of degree at most $t+p^{\lfloor \log_p t \rfloor}$ which contains $A$.

If the degree of $\phi$ is zero then there must be at least two co-prime polynomials of degree at most $t+p^{\lfloor \log_p t \rfloor}$ both of which are zero on $A$ and we have conclusion (i).

If the degree of $\phi$ is not zero then the above equation implies that 
$$
\phi^*(Y)=cY_1^{i}Y_2^j  \pmod {Y_3},
$$
for some integers $i,j$ such that $i+j=\deg \phi^*=\deg \phi$ and some $c \in {\mathbb F}_q$.
With respect to the canonical basis this gives
$$
\phi(Y)=\alpha(Y)^i \beta(Y)^j  \pmod {\gamma (Y)},
$$
where the kernel of the linear form $\gamma$ is the line joining $x$ and $y$ and $\alpha$ and $\beta$ are linear forms with the property that $\alpha(y)=0$ and $\beta(x)=0$.
Thus, we have proved that if the kernel of a linear form $\gamma$ is a bisecant to $S$ then $\phi$ modulo $\gamma$ factorises into at most two linear factors whose multiplicities sum to the degree of $\phi$.

Therefore, $\phi$ is hyperbolic on $S$ and we have conclusion (ii).
\end{proof}

\section{Proof of Theorem~\ref{twocurves}.}

\begin{lemma} \label{itisitis}
If there is a homogeneous polynomial $\phi$ which is hyperbolic on an arc $S$, where $|S| \geqslant 2\deg \phi+2$, then all but at most one point of $S$ are contained in a conic and if $q$ is odd then $S$ is contained in a conic.
\end{lemma}

\begin{proof} 
Let $r$ be the degree of $\phi$. 

Suppose there is a point $x$ in $S$ which is not in the zero set of $\phi$. If $\gamma$ is a linear form such that the kernel of $\gamma$ is a bisecant joining $x$ to a point $y$ of $S$ then, since $\phi$ is hyperbolic on $S$, $\phi(X)=\alpha(X)^r$ modulo $\gamma$ for some linear form $\alpha$, where $\alpha(y)=0$. This implies that $\phi(y)=0$, so $\phi$ is zero on all but at most one point of $S$.

Observe that $r \geqslant 2$, since $S$ is not contained in a line. Also, we can assume that $\phi$ is not a $p$-th power, since we can replace $\phi$ by its $p$-th root and all the properties are preserved.

Choose a suitable basis so that $(1,0,0)$, $(0,1,0)$ and $(0,0,1)$ are points of $S \setminus \{ x\}$. 

Suppose every term of $\phi(X)$ is of the form $c_{ijk}X_1^iX_2^{jp}X_3^{kp}$. Since $\phi(X)$ is hyperbolic on $S$, $\phi(0,X_2,X_3)=c_{0jk}X_2^{jp}X_3^{kp}$, for some $j$ and $k$. Hence $r=(j+k)p$. Considering any term with $i>0$, it follows that $p$ divides $i$. So $\phi$ is a $p$-th power, which it is not. Hence we can assume, without loss of generality, that some exponent of $X_1$ is not a multiple of $p$, some exponent of $X_3$ is not a multiple of $p$ and that the degree of $\phi$ in $X_3$ is at most the degree of $\phi$ in $X_1$.

Let $n$ be the degree of $\phi$ in $X_1$. 
Suppose $n \geqslant 2$. 

Then $2 \leqslant n \leqslant r-1$.

Write 
\begin{equation} \label{hypeqn1}
\phi(X)=\sum_{j=0}^n X_1^{n-j} c_j(X_2,X_3),
\end{equation} 
where $c_j$ is either zero or a homogeneous polynomial of degree $r-n+j$ and by assumption $c_0(X_2,X_3) \neq 0$. 

Let 
$$
E=\{ e \in {\mathbb F}_q \ | \ X_3-eX_2 \ \mathrm{is\ a\ bisecant\ to}\ S \ \mathrm{and}\ c_0(X_2,eX_2) \neq 0 \}.
$$
Since there are $|S|-2$ bisecants to $S$ of the form $X_3-eX_2=0$ and $c_0(X_2,eX_3)=0$ for at most $r-n$ values of $e \in {\mathbb F}_q$, we have that 
$$
|E| \geqslant |S|-2-r+n \geqslant r+n.
$$
Since $\phi(X)$ is hyperbolic on $S$, for all $e \in E$, there exists a $d$ such that the point $(-d,1,e) \in S$ and $\phi(X)$ modulo $X_3-eX_2$ factorises into the product of two linear factors. One of these factors is $X_1+dX_2$, while the other does not contin an $X_1$ term, since it corresponds to the point $(1,0,0)$. Hence,
\begin{equation} \label{hypeqn2}
\phi(X_1,X_2,eX_2)=(X_1+dX_2)^n c_0(X_2,eX_2).
\end{equation} 

Comparing the coefficient of $X_1^{n-j}$ in (\ref{hypeqn1}) and (\ref{hypeqn2}) we have that for $j=1,\ldots,n$,
$$
c_j(X_2,eX_2)={n \choose j} d^j X_2^j c_0(X_2,eX_2).
$$

If $p$ divides $n$ and $j$ is not a multiple of $p$ then $c_j(X_2,eX_2)=0$ for all $e \in E$. Since the degree of $c_j$ is at most $r$ and $r \leqslant |E|-1$, $c_j(X_2,X_3)=0$. But then this implies that each exponent of $X_1$ in a term of $\phi(X)$ is a multiple of $p$, a contradiction. Therefore, $n$ is not a multiple of $p$.

For $e \in E$ and $j=1$ we have that 
$$
c_1(X_2,eX_2)=ndX_2c_0(X_2,eX_2).
$$
Thus, for $j=1,\ldots,n$, substituting for $d$ we obtain
$$
c_0(X_2,eX_2)^{j-1}c_j(X_2,eX_2)n^j={n \choose j} c_1(X_2,eX_2)^j.
$$
Hence, $h_j(e)=0$, for all $e \in E$, where $h_j(Y)$ is a polynomial in $({\mathbb F}_q[X_2])[Y]$ defined by
$$
h_j(Y)=c_0(X_2,YX_2)^{j-1}c_j(X_2,YX_2)n^j-{n \choose j} c_1(X_2,YX_2)^j.
$$

Let $m$ be the degree of the polynomial $h_2(Y)$. Then $m \leqslant 2(r-n+1)$ since the degree of $c_j$ is $r-n+j$, and $m\leqslant 2n$ since the degree of $c_j(X_2,X_3)$ in $X_3$ is at most $n$. 

To be able to conclude that $h_2$ is identically zero, we need $m \leqslant |E|-1$.

If $m \geqslant |E|$ then 
$$
2(r-n+1) \geqslant m \geqslant |E| \geqslant  r+n,
$$ 
so $r-n+2 \geqslant 2n \geqslant m$, contradicting the assumption that $m \geqslant |E|$.

Therefore, $h_2(Y)$ is identically zero. This implies that the polynomial $c_0(X_2,YX_2)$ divides $c_1(X_2,YX_2)$ and so
$$
c_1(X_2,YX_2)=(aX_2+bYX_2)c_0(X_2,YX_2),
$$ 
for some $a,b \in {\mathbb F}_q$.
Hence,
$$
h_j(Y)=c_0(X_2,YX_2)^{j-1}(c_j(X_2,YX_2)n^j-{n \choose j} c_0(X_2,YX_2)(aX_2+bYX_2)^{j}).
$$

But for each $e \in E$, the polynomial 
$$c_j(X_2,YX_2)n^j-{n \choose j} c_0(X_2,YX_2)(aX_2+bYX_2)^{j},$$ 
is zero at $Y=e$. It has degree at most $r-n+j\leqslant r \leqslant |E|-1$ in $Y$, so we conclude that it is identically zero.

Substituting $Y=X_3/X_2$, we have that
$$
c_j(X_2,X_3)n^j={n \choose j} c_0(X_2,X_3)(aX_2+bX_3)^{j},
$$
for $j=1,\ldots,n$.

Hence,
$$
\phi(X)=\sum_{j=0}^n {n \choose j} X_1^{n-j}  c_0(X_2,X_3)(\frac{a}{n}X_2+\frac{b}{n}X_3)^{j},
$$
and therefore
$$
\phi(X)=c_0(X_2,X_3)(X_1+\frac{a}{n}X_2+\frac{b}{n}X_3)^n.
$$
Since $\phi$ is zero at the points of $S \setminus \{ x \}$, this implies that all but $r-n$ points of $S \setminus \{ x \}$ are contained in a line, which gives $|S|-1 \leqslant r-n+2 \leqslant r$, a contradiction.

Hence, $n=1$. Since the degree of $\phi$ in $X_3$ is at most the degree of $\phi$ in $X_1$, this implies that the degree of $\phi$ in $X_3$ is one. Since $\phi$ is hyperbolic on $S$, $\phi(X_1,0,X_3)$ is a constant times $X_1X_3$. Therefore, $r=2$ and we have proved that $\phi$ is a quadratic form and that $S \setminus \{ x\}$ is contained in a conic.

Now, let $\gamma$ be a linear form whose kernel is a bisecant to $S$ containing $x$. Then, as in the first paragraph this bisecant to $S$ is a tangent to $\phi$. Since a point not in a conic is on at most two tangents to a conic, when $q$ is odd, and $|S \setminus \{ x \}| \geqslant 3$ by assumption, we conclude that no such point $x$ exists and that $S$ is contained in a conic if $q$ is odd.
\end{proof}

We now prove Theorem~\ref{twocurves}, which we restate here for convenience.

{\bf Theorem \ref{twocurves}.}
{\em
Let $A$ be a planar arc of size $q+2-t$, $q$ odd, not contained in a conic. 

If $A$ is not contained in a curve of degree $t$ then it is contained in the intersection of two curves of degree at most $t+p^{\lfloor \log_p t \rfloor}$ which do not share a common component. 

If $A$ is contained in a curve $\phi$ of degree $t$ and
\begin{equation}
p^{\lfloor \log_p t \rfloor}(t+\tfrac{1}{2}p^{\lfloor \log_p t \rfloor}+\tfrac{3}{2}) \leqslant \tfrac{1}{2}(t+2)(t+1)
\end{equation}
then there is another curve of degree at most $t+p^{\lfloor \log_p t \rfloor}$ which contains $A$ and shares no common component with $\phi$.
}

\begin{proof} (of Theorem~\ref{twocurves}.)
Suppose that $A$ is contained in a curve of degree $t$ and that
$$
p^{\lfloor \log_p t \rfloor}(t+\tfrac{1}{2}p^{\lfloor \log_p t \rfloor}+\tfrac{3}{2}) \leqslant \tfrac{1}{2}(t+2)(t+1).
$$
By Lemma~\ref{tplusrsocle}, there is a subset $S''$ of $A$ which contains a $(t+j)$-socle $S_j$, for all $j=0,\ldots, p^{\lfloor \log_p t \rfloor}$. Moreover,
$$
|S'' \setminus S_0| \leqslant p^{\lfloor \log_p t \rfloor}(t+\tfrac{1}{2}p^{\lfloor \log_p t \rfloor}+\tfrac{3}{2}) \leqslant \tfrac{1}{2}(t+2)(t+1).
$$
We can extend $S''$ to a subset $S$ of $A$ such that
$$
|S \setminus S_0| = \min \{  \tfrac{1}{2}(t+2)(t+1), |A \setminus S_0| \}.
$$
Suppose there are not two co-prime homogeneous polynomials of degree at most $t+p^{\lfloor \log_p t \rfloor}$ which are zero on $A$. Then, by Theorem~\ref{main} and Lemma~\ref{almost}, there is a homogeneous polynomial $\phi$ of degree at most $t+p^{\lfloor \log_p t \rfloor}$ which is hyperbolic on $S$.

Suppose $|S\setminus S_0|=|A\setminus S_0|$ then $S=A$. If $|S| \leqslant 2\deg \phi+1$ then we can trivially cover $A$ with one conic and bisecants in two ways so as to give two co-prime homogeneous polynomials of degree at most $\deg \phi$ which are zero on $A$, contrary to the supposition. Therefore, $|S| \geqslant 2\deg \phi+2$.

Suppose $|S\setminus S_0|=\frac{1}{2}(t+2)(t+1)$.

If $p^{\lfloor \log_p t \rfloor} \geqslant p$ then
$$
|S| \geqslant \tfrac{1}{2}(t+2)(t+1) \geqslant p^{\lfloor \log_p t \rfloor}(t+\tfrac{1}{2}p^{\lfloor \log_p t \rfloor}+\tfrac{3}{2}) \geqslant 2(t+p^{\lfloor \log_p t \rfloor})+2\geqslant 2 \deg \phi+2.
$$ 
And if $p^{\lfloor \log_p t \rfloor}=1$ then
$$
|S| \geqslant \tfrac{1}{2}(t+2)(t+1)+t \geqslant  2(t+1)+2\geqslant 2 \deg \phi+2.
$$ 
Thus, in all cases, $|S| \geqslant 2\deg \phi+2$, so by Lemma~\ref{itisitis}, $S$ is contained in a conic $C$. 

If there is a point $x$ of $A$ which is not in $C$. We can re-define $S$ to contain $x$ and five points of $A \cap C$ and get a contradiction. Therefore, $A$ is contained in the conic $C$ which, by hypothesis, it is not.

Therefore, there are two co-prime homogeneous polynomials of degree at most $t+p^{\lfloor \log_p t \rfloor}$ which are zero on $A$. By Lemma~\ref{threepolys}, $A$ is contained in a curve of degree $t$ and a curve of degree at most $t+p^{\lfloor \log_p t \rfloor}$ which do not share a common component.

Suppose $A$ is not contained in a curve of degree $t$. Then
$$
|A| \geqslant {t+2 \choose 2},
$$
which implies
$$
|A| \geqslant 2(t+p^{\lfloor \log_p t \rfloor})+2
$$
unless $t \leqslant 4$, $p=3$ and $q \geqslant 27$. We can assume $t \geqslant 3$, since Segre \cite{Segre1967} proved that an arc of size $q$, $q$ odd, is contained in a conic.

However, for $q \geqslant 27$ and $3 \leqslant t \leqslant 4$,
$$
|A|=q+2-t  \geqslant 2(t+p^{\lfloor \log_p t \rfloor})+2.
$$
Therefore, we may assume,
$$
|A| \geqslant 2(t+p^{\lfloor \log_p t \rfloor})+2.
$$

Suppose there are not two co-prime homogeneous polynomials of degree at most $t+p^{\lfloor \log_p t \rfloor}$ which are zero on $A$. Then, by Lemma~\ref{almost}, there is a homogeneous polynomial $\phi$ of degree at most $t+p^{\lfloor \log_p t \rfloor}$ which is hyperbolic on $A$. By Lemma~\ref{itisitis}, $A$ is contained in a conic, contradicting the assumption that it is not contained in a curve of degree $t$.
\end{proof}
 
\section{Planar arcs in planes of order less than $32$ of odd characteristic.}

The entries in the following array are the number of complete arcs in planes of odd order less than $32$, for which $|A|=q+2-t \geqslant 2d+1$, where $d=t+ p^{\lfloor \log_p t \rfloor}$. Again, we have used \cite[Table 9.4]{Hirschfeld1998} and the articles \cite{Coolsaet2015}, \cite{CS2009} and \cite{CS2011} to compile this data. Below these values the trivial construction in the proof of Theorem~\ref{twocurves} suffices to prove that $A$ is contained in the intersection of two curves of degree at most $d$, sharing no common component.

$$
\begin{array}{c|cccccccccccc}
\hline
 q & 5 & 7 & 9 & 11 & 13 & 17 & 19 & 23 & 25 & 27 & 29 & 31 \\
|A| \geqslant & 6 & 7 & 10 & 10 & 11 & 14 & 15 & 18 & 23 & 24 & 22 & 23   \\
\# \ \mathrm{not\ contained\ in\ a\ conic} & 0 & 0 & 0 & 1 & 1 & 1 & 0 & 0 & 0 & 0 & 1 & 0\\
\hline
\end{array}
$$

We consider each of the four complete arcs of size at least $2d+1$ in planes of odd characteristic of order at most $31$. In each case, we construct two curves of degree $d$, sharing no common component, containing the arc, whose existence follows from Theorem~\ref{twocurves}. 

The unique arc $A$ of size $10$ in $\mathrm{PG}(2,11)$ is Example~\ref{10arc}. Since any five points lie on a conic, we can cover the points of $S$ by two conics $C_1$ and $C_1'$. By choosing four points on $C_1$ and one point on $C_1'$, we can cover at least these five points of $A$ by another conic $C_2$ and therefore all the points of $A$ by two conics $C_2$ and $C_2'$. This gives us two quartic curves which share no common component, both of which contain $A$. 

The unique arc of size $12$ in $\mathrm{PG}(2,13)$ is Example~\ref{12arc}. As we already saw in Example~\ref{12arc} it is contained in the intersection of the two quartic curves given by the equations $x_2^4=x_1^2x_3^2$ and $x_1^4=x_3^2x_2^2$.

The unique arc $A$ of size $14$ in $\mathrm{PG}(2,17)$ has $10$ points on a conic $C$. Hence, we get one curve of degree four covering $A$ by choosing two lines covering the points of $A \setminus C$. Now, let $C_1$ be a conic through four points of $A \cap C$ and one point of $A \setminus C$ and $C_2$ be a conic through four other points of $S \cap C$ and one other point of $A \setminus C$. We can cover the points of $A \setminus (C_1 \cup C_2)$ with two lines and obtain a curve of degree $6$ which contains $A$ and does not share a common component with the curve of degree $4$ containing $A$.

The unique arc of size $24$ in $\mathrm{PG}(2,29)$ is Example~\ref{24arc}. It is exactly the zero set of the polynomial 
$$
g(x)=x_1^3x_2+x_2^3x_3+x_3^3x_1.
$$
As observed before it is also contained in a curve of degree $7$ which shares no component with the curve defined by $g$.

{\small  Simeon Ball}\\
 {\small  Departament de Matem\`atiques,} \\
{\small Universitat Polit\`ecnica de Catalunya,} \\
{\small M\`odul C3, Campus Nord,}\\
{\small c/ Jordi Girona 1-3,}\\
{\small 08034 Barcelona, Spain} \\
 {\tt simeon@ma4.upc.edu} \\

{\small   Michel Lavrauw}\\
{\small Faculty of Engineering and Natural Sciences,}\\
{\small Sabanc\i \  University,}\\
{\small Istanbul, Turkey}\\
  {\tt mlavrauw@sabanciuniv.edu}}



\end{document}